\documentclass[12pt]{amsart}
\usepackage{geometry}                
\usepackage{graphicx}
\usepackage{amssymb,enumerate}
\usepackage{amsmath,amsfonts,amsthm}
\usepackage{exscale}
\usepackage{latexsym}
\usepackage{pict2e}
\usepackage{verbatim}
\usepackage{epstopdf}
\usepackage[all]{xy}
\theoremstyle{plain}
\newtheorem{thm}{Theorem}[section]
\newtheorem*{main}{Theorem}
\newtheorem*{lem}{Lemma}
\newtheorem{prop}[thm]{Proposition}

\newtheorem{ch}{Characterization}

\theoremstyle{definition}
\newtheorem*{ex}{Example}
\newtheorem{defn}[thm]{Definition}
\newtheorem{rmk}[thm]{Remark}

\newcommand{\bb}{bigrassmannian }

\begin{document}


\title[Enumeration of Bigrassmannian permutations]{Enumeration of bigrassmannian permutations below a permutation in Bruhat order}
\author{Masato Kobayashi}
\date{\today}                                           
\address{Department of Mathematics\\
the University of Tennessee, Knoxville, TN 37996}
\email{kobayashi@math.utk.edu}
\thanks{This article will appear in Order (published online on May 13, 2010). The final publication is available at \textsf{springerlink.com}.}

\maketitle
\begin{abstract}In theory of Coxeter groups, bigrassmannian elements are well known as elements which have precisely one left descent and precisely one right descent. In this article, we prove formulas on enumeration of bigrassmannian permutations weakly below a permutation in Bruhat order in the symmetric groups. For the proof, we use equivalent characterizations of bigrassmannian permutations by Lascoux-Sch\"{u}tzenberger and Reading.
\end{abstract}

\section{Introduction}

In the theory of Coxeter groups, bigrassmannian elements are known as elements which have precisely one left descent and precisely one right descent. They play a significant role to investigate structure of the Bruhat order \cite{geck}. In particular, in the symmetric group (type A), bigrassmannian permutations have many nice order-theoretic properties. First, Lascoux-Sch\"{u}tzenberger proved \cite{lascoux} that a permutation is bigrassmannian if and only if it is join-irreducible. For definition of join-irreduciblity, see \cite[Sections 2]{reading4}. Second, Reading \cite{reading4} characterized join-irreducible permutations as certain minimal monotone triangles. \\ \indent In this article, we will make use of these characterizations to answer the following question:\,given a permutation $x$, how can we find the number of bigrassmannian permutations weakly below it in Bruhat order? Unfortunately, this is not easy from the usual definition of Bruhat order. Instead, it is much easier to use monotone triangles because the set of monotone triangles has a partial order which is equivalent to Bruhat order over the symmetric groups. Moreover, there is a natural identification of join-irreducible (equivalently, bigrassmannian) permutations with entries of monotone triangles. We will see detail of these in Section 2. In Section 3, we prove the main result:
\begin{main}
For $x\in S_n$, let $I(x)$ be the set of inversions of $x$ and $\beta(x)$ the number of \bb  permutations weakly below $x$ in the Bruhat order. Then we have
\begin{align*}
\beta(x)=\sum_{a=1}^{n-1} (x(a)-a)(n-a)=\frac{1}{2}\sum_{a=1}^n (x(a)-a)^2=\sum_{(i,\,j)\in I(x)} (x(i)-x(j)).
\end{align*}
\end{main}

\section{Two characterizations of bigrassmannian permutations}

We begin with definition of the Bruhat order.

\begin{defn}\label{br}{Let $x\in S_n$.
A pair of integers $(i, j)$ is said to be an \emph{inversion} of $x$ if $1\le i<j\le n$ and $x(i)>x(j)$.
Let $I(x)$ denote the set of all inversions. Define the length $\ell(x)$ to be $\#I(x)$.
Let $t_{ij}$ denote a transposition $(i<j)$. In particular, write $s_i=t_{i, i+1}$. It is well-known that $S=\{s_1, \dots, s_{n-1}\}$ generates $S_n$ and $\ell(x)$ is equal to the minimum number $k$ such that $x=s_{i_1}s_{i_2}\dots s_{i_k}$ (the identity permutation $e$ has length 0 with the empty word). A \emph{reduction} of $x$ is a permutation of the form 
$xt_{ij}$ with $(i, j)\in I(x)$. Define the Bruhat order on $S_n$ as $w\le y$ if  there exist $x_0, x_1, \dots, x_k\in S_n$ such that $x_0=w, x_k=y$ and $x_{i}$ is a reduction of $x_{i+1}$ for all $0\le i\le k-1$.}\end{defn}


\begin{defn}
For $x \in S_n$, define \emph{left} and \emph{right descents} to be 
\begin{align*}
D_L(x)&=\{s_i\in S \mid x^{-1}(i)>x^{-1}(i+1)\},\\
D_R(x)&=\{s_i\in S \mid x(i)>x(i+1)\}.
\end{align*}
We say that $x$ is \emph{bigrassmannian} if $\#D_L(x)=\#D_R(x)=1$. Define 
\begin{align*}
B(x)&=\{w \mid w\le x \mbox{ and  $w$ is bigrassmannian}\},\\
\beta(x)&=\#B(x).
\end{align*}
\end{defn}

We would like to know $B(x)$ and $\beta(x)$ for a given $x$. As mentioned earlier, this is not easy from the definition of Bruhat order. However, the following two equivalent characterizations of bigrassmannian permutations by Lascoux-Sch\"{u}tzenberger and Reading are helpful.

\begin{ch}\cite[Th\'{e}or\`{e}me 4.4]{lascoux} \label{33}$x\in S_n$ is bigrassmannian if and only if it is join-irreducible (Lascoux and Sch\"{u}tzenberger used terminology the \emph{bases} rather than the set of join-irreducible elements).
\end{ch}

Before Reading's characterization, let us see the definition of monotone triangles.

\begin{defn}\label{monot}
 A \emph{monotone triangle} $x$ of order $n$ is an $n(n-1)/2$-tuple $(x_{ab}\: | \: 1\le b\le a\le n-1)$ such that $1\le x_{ab} \le n, x_{ab}< x_{a,b+1}, x_{ab}\ge x_{a+1, b}$ and $x_{ab}\le x_{a+1,b+1}$ for all $a, b$.  Regard a permutation $x\in S_n$ as a monotone triangle of order $n$ as follows:
for each $1\le a\le n-1$, let $x_{a1}, x_{a2}, \dots, x_{aa}$ be integers such that $\{x(1), x(2), \dots, x(a)\}=\{x_{a1}, x_{a2}, \dots, x_{aa}\}, x_{ab}<x_{a,b+1}$ for all $1\le b\le a-1$. Then $x=(x_{ab})$ is a monotone triangle. Denote by $L(S_n)$ the set of all monotone triangles of order $n$. Define a partial order on $L(S_n)$ by $x\le y$ if $x_{ab} \le y_{ab}$ for all $a, b$.
\end{defn}

Following \cite[Section 8]{reading4}, we introduce an important family of monotone triangles.

\begin{defn}\label{mini}For positive integers $(a, b, c)$ such that $1\le b\le a\le n-1$ and $b+1\le c\le n-a+b$, define $J_{abc}$ to be the componentwise smallest monotone triangle such that $a, b$ entry is $\ge c$ (notice that Reading worked on $S_{n+1}$($=$ Coxeter group of type $\textnormal{A}_n$) while here we are working on $S_{n}$). In other words, $J_{abc}$ satisfies $x\ge J_{abc}$ if and only if $x_{ab}\ge c$ for $x\in L(S_n)$.
\end{defn}

\begin{ch}\cite[Section 8]{reading4}
$x\in S_n$ is join-irreducible if and only if there exist some $(a, b, c)$ with $1\le b\le a\le n-1$ and $b+1\le c\le n-a+b$ such that $x=J_{abc}$.
\end{ch}

As a consequence of minimality of $J_{abc}$ in Definition \ref{mini}, it is easy to compare join-irreducible monotone triangles at the same position $(a, b)$ as \[J_{abc}<J_{abd} \iff c<d\]
for all $c, d$ with $b+1\le c, d\le n-a+b$.
Hence we may identify entries appearing in ($x_{ab}$) with ``($J_{abx_{ab}})$". This identification is quite useful to find $\beta(x)$ (because of Characterizations 1 and 2) as we shall see in Proposition \ref{p}.

\begin{rmk}
In fact, $L(S_n)$ is a distributive lattice and the MacNeille completion of $S_n$ (meaning smallest lattice which contains $S_n$). In particular, for all $x, y\in S_n$, $x\le y$ in Bruhat order (Definition \ref{br}) is equivalent to $x\le y$ as monotone triangles (Definition \ref{monot}).
Since join-irreducible elements are invariant under the MacNeille completion, even for $x\in L(S_n)$, $\beta(x)$ makes sense as the number of join-irreducible monotone triangles weakly below $x$. For detail, see \cite{balcza,bjorner}, \cite[Th\'{e}or\`{e}me 4.4]{lascoux} and \cite[Sections 6, 7, 8]{reading4}.
\end{rmk}

\begin{defn}
For $x\in L(S_n)$, define
\begin{align*}
\Sigma(x)&= \sum_{a=1}^{n-1} \sum_{b=1}^{a} x_{ab}.
\end{align*}
\end{defn}

\begin{prop}\label{p} \hfill
\begin{enumerate}
\item For each $a, b$ such that $1\le b\le a\le n-1$, there is a chain of \bb permutaitons:
\begin{align*}
J_{a, b, b+1} &< J_{a,b,b+2} < \dots < J_{a, b, n-a+b}.
\end{align*}
Consequently for $x\in L(S_n)$,
\begin{align*}
 \underbrace{J_{a, b, b+1}, J_{a, b, b+2}, \dots, J_{abx_{ab}}}_{x_{ab}-b}\in B(x).
\end{align*}
\item Let $x\in L(S_n)$. Then $\beta(x)=\Sigma(x)-\Sigma(e)$.
\end{enumerate}
\end{prop}

\begin{proof}\hfill
\begin{enumerate}
\item Use $J_{abc}<J_{abd} \iff c<d$. 
\item Note that $e_{ab}=b$ for all $a, b$. It then follows from (1) that 
\begin{align*}
\beta(x)&=\{ w\in S_n \mid w\le x \mbox{ and } w \mbox{ is bigrassmannian}\}\\
&=\sum_{a=1}^{n-1}\sum_{b=1}^{a}\#\{J_{abc} \mid b+1\le c\le x_{ab}\}\\
&=\sum_{a=1}^{n-1}\sum_{b=1}^{a}(x_{ab}-b)\\
&=\Sigma(x)-\Sigma(e).
\end{align*}
\end{enumerate}
\end{proof}

\section{Proof of Theorem}

We saw the formula of $\beta(x)$ for general monotone triangles $x$.
If $x$ is a permutation, there are simpler formulas of $\beta(x)$ because $x(a)$ appears $n-a$ times in entries of the monotone triangle for each $a$ so that it is easier to compute $\Sigma(x)$.

\begin{main}
For all $x\in S_n$, we have
\begin{align*}
\beta(x)=\sum_{a=1}^{n-1} (x(a)-a)(n-a)=\frac{1}{2}\sum_{a=1}^n (x(a)-a)^2=\sum_{(i,\,j)\in I(x)} (x(i)-x(j)).
\end{align*}
\end{main}

\textit{Proof.} We show the first equality.
\begin{align*}
\sum_{a=1}^{n-1} (x(a)-a)(n-a)&=\sum_{a=1}^{n-1} x(a)(n-a)-\sum_{a=1}^{n-1} a(n-a)\\
&=\Sigma(x)-\Sigma(e)=\beta(x).
\end{align*}
Next we check the second equality. 
Since \begin{align*}\sum_{a=1}^nx(a)=\sum_{a=1}^n a \mbox{\quad and \quad }
\sum_{a=1}^nx(a)^2=\sum_{a=1}^n a^2,
\end{align*} we have 
\begin{align*}
\frac{1}{2}\sum_{a=1}^n (x(a)-a)^2&=\frac{1}{2}\sum_{a=1}^n (x(a)^2-2ax(a)+a^2)\\
&=\sum_{a=1}^n (a^2- ax(a))\\
&=\sum_{a=1}^n (a^2- ax(a)+x(a)n-an)\\
&=\sum_{a=1}^n x(a)(n-a)- \sum_{a=1}^n a(n-a)\\
&=\sum_{a=1}^{n-1} x(a)(n-a)- \sum_{a=1}^{n-1} a(n-a)\\
&=\Sigma(x)-\Sigma(e)\\
&=\beta(x).
\end{align*}

Before the proof the last equality, we need a lemma.

\begin{lem}\label{l1}
Let $x\in S_n$ and $i<j$. Then we have 
\begin{align*}
\beta(x)-\beta(xt_{ij})=(j-i)(x(i)-x(j)).
\end{align*}
In particular, $\beta(x)-\beta(xs_i)=x(i)-x(i+1)$.
\end{lem}
\textit{Proof.}
Let $w=xt_{ij}$.
Note that $w(i)=x(j), w(j)=x(i)$ and $w(a)=x(a)$ for all $a\ne i, j$. Then apply the first equality as just shown to $w$ and $x$:
\begin{align*}
\beta(x)-\beta(w)&=\sum_{a=1}^{n-1} (x(a)-a)(n-a)-\sum_{a=1}^{n-1} (w(a)-a)(n-a)\\
&=\sum_{a=1}^{n} (x(a)-a)(n-a)-\sum_{a=1}^{n} (w(a)-a)(n-a)\\
&=\sum_{a=1}^{n} (x(a)-w(a))(n-a)\\
&=(x(i)-w(i))(n-i)+(x(j)-w(j))(n-j)\\
&=(x(i)-x(j))(n-i)-(x(i)-x(j))(n-j)\\
&=(j-i)(x(i)-x(j)). \, \blacksquare
\end{align*}
  
\begin{proof}[Proof of the last equality]
The proof is induction on $\ell(x)$. If $\ell(x)=0$, then $x=e$ and hence $\beta(e)=0$.
If $\ell(x)>0$, we can choose some $a$ such that $(a, a+1) \in I(x)$ (otherwise $x=e$ since $x(1)<x(2)<\cdots <x(n)$). Let $w=xs_a$. Note that $(a, a+1)\notin I(w)$. Now set
\begin{align*}
I_1(w)&=\{(i, a)\in I(w) \: | \: 1\le i\le a-1 \},\\
I_2(w)&=\{(i, a+1)\in I(w) \: | \: 1\le i\le a-1 \},\\
I_3(w)&=\{(a, j)\in I(w) \: | \: a+2\le j\le n \},\\
I_4(w)&=\{(a+1, j)\in I(w) \: | \: a+2\le j\le n \},\\
I_5(w)&=\{(i, j)\in I(w) \: | \: i,j \not\in \{a, a+1\} \}.
\end{align*}
Clearly $I(w)=\cup I_p(w)$ and the union is disjoint. 
Observe that $(i,a)\in I_1(w) \Longleftrightarrow (i, a+1)\in I_2(x)$ since $w(a)<w(i) \iff x(a+1)< x(i)$ for $1\le i \le a-1$. Therefore
\begin{align*}
\sum_{(i,\, a)\in I_1(w)} (w(i)-w(a))=\sum_{(i,\, a+1)\in I_2(x)}  (x(i)-x(a+1)).
\end{align*}
It is quite similar to show that  
\begin{align*}
(i, a+1)\in I_2(w) &\Longleftrightarrow (i, a)\in I_1(x),\\ 
(a, j)\in I_3(w) &\Longleftrightarrow (a+1, j)\in I_4(x),\\ 
(a+1, j)\in I_4(w) &\Longleftrightarrow (a, j)\in I_3(x),\\ 
(i, j)\in I_5(w) &\Longleftrightarrow (i, j)\in I_5(x).
\end{align*}
Since $\ell(w)=\ell(x)-1$, the hypothesis of induction tells us that
\begin{align*}
\beta(w)=\sum_{(i, \, j)\in I(w)} (w(i)-w(j)).
\end{align*}
Then thanks to the Lemma, we conclude that 
\begin{align*}
\beta(x)&=\beta(w)+(x(a)-x(a+1))\\
&=\sum_{(i,\,j) \in I(w)}(w(i)-w(j)) +(x(a)-x(a+1))\\
&=\sum_{p=1}^5 \sum_{I_p(x)} (x(i)-x(j))+(x(a)-x(a+1))\\
&=\sum_{(i, \, j)\in I(x)} (x(i)-x(j)).
\end{align*}
\end{proof}



\begin{ex}
Let $x=42513$. Then
\begin{align*}
\Sigma \left(
\begin{matrix}
4&&&\\
2&4&&\\
2&4&5&\\
1&2&4&5
\end{matrix}
\right)
\ - \
\Sigma \left(
\begin{matrix}
1&&&\\
1&2&&\\
1&2&3&\\
1&2&3&4
\end{matrix}
\right)
= \
\Sigma \left(
\begin{matrix}
3&&&\\
1&2&&\\
1&2&2&\\
0&0&1&1
\end{matrix}
\right)
=13,
\end{align*}
\begin{align*}
\frac{1}{2}((x(1)-1)^2+(x(2)-2)^2+(x(3)-3)^2+(x(4)-4)^2+(x(5)-5)^2)=13,\\
(x(1)-1)4+(x(2)-2)3+(x(3)-3)2+(x(4)-4)1=13,
\end{align*}
\end{ex}
\begin{align*}
\begin{split}
\sum_{(i,\,j)\in I(x)}(x(i)-x(j))&=x(1)-x(2)+x(1)-x(4)+x(1)-x(5)+x(2)-x(4)\\
&\quad {}+x(3)-x(4)+x(3)-x(5)=13.
\end{split}
\end{align*}

\vspace{1ex}
\begin{center}
\textbf{Acknowledgment.}\\
\end{center}
The author would like to thank the anonymous referee for helpful comments and suggestions.

\nocite{reading4,balcza,bjorner,geck}
\bibliography{Kobayashi_enbigr}

\providecommand{\bysame}{\leavevmode\hbox to3em{\hrulefill}\thinspace}
\providecommand{\MR}{\relax\ifhmode\unskip\space\fi MR }
\providecommand{\MRhref}[2]{%
  \href{http://www.ams.org/mathscinet-getitem?mr=#1}{#2}
}
\providecommand{\href}[2]{#2}
\begin{thebibliography}{1}

\bibitem{balcza}
L.~Balcza, \emph{Sum of lengths of inversions in permutations}, Discrete Math.
  \textbf{111} (1993), no.~1-3, 41--48.

\bibitem{bjorner}
A.~Bj{\"{o}}rner and F.~Brenti, \emph{An improved tableau criterion for
  {Bruhat} order}, Electron. J. Combin. \textbf{3\: \textnormal{\#R22}} (1996),
  no.~1, 5pp.

\bibitem{geck}
M.~Geck and S.~Kim, \emph{Bases for the {Bruhat}-{Chevalley} order on all
  finite {Coxeter} groups}, J. Algebra \textbf{197} (1997), no.~1, 278--310.

\bibitem{lascoux}
A.~Lascoux and M.-P. Sch{\"{u}}tzenberger, \emph{Treillis et bases des groupes
  de {Coxeter}}, Electron. J. Combin. \textbf{3\: \textnormal{\#R27}} (1996),
  35pp (French).

\bibitem{reading4}
N.~Reading, \emph{Order dimension, strong {Bruhat} order and lattice propeties
  for posets}, Order \textbf{19} (2002), no.~1, 73--100.

\end{thebibliography}
\bibliographystyle{amsplain}

\end{document}